\documentclass[12pt, reqno]{amsart}
\usepackage{amsmath, amsthm, amscd, amsfonts, amssymb, graphicx, color}
\usepackage[bookmarksnumbered, colorlinks, plainpages]{hyperref}
\hypersetup{colorlinks=true,linkcolor=red, anchorcolor=green, citecolor=cyan, urlcolor=red, filecolor=magenta, pdftoolbar=true}

\newtheorem{theorem}{Theorem}[section]
\newtheorem{lemma}[theorem]{Lemma}
\newtheorem{proposition}[theorem]{Proposition}
\newtheorem{cor}[theorem]{Corollary}
\theoremstyle{definition}

\theoremstyle{remark}
\newtheorem{remark}[theorem]{\bf{Remark}}
\numberwithin{equation}{section}
\begin{document}

\title [Numerical radius inequalities for tensor product of operators]  {Numerical radius inequalities for tensor product of operators}
	
	\author[ A. Sen, P. Bhunia, K. Paul] {Anirban Sen, Pintu Bhunia, Kallol Paul}

\address [Sen] {Department of Mathematics, Jadavpur University, Kolkata 700032, West Bengal, India}
\email{anirbansenfulia@gmail.com}

\address [Bhunia] {Department of Mathematics, Jadavpur University, Kolkata 700032, West     Bengal, India}
\email{pintubhunia5206@gmail.com}
\email{pbhunia.math.rs@jadavpuruniversity.in}

\address[Paul] {Department of Mathematics, Jadavpur University, Kolkata 700032, West Bengal, India}
\email{kalloldada@gmail.com}
\email{kallol.paul@jadavpuruniversity.in}

\thanks{Mr. Anirban Sen would like to thank CSIR, Govt. of India for the financial support in the form of Junior Research Fellowship under the mentorship of Prof Kallol Paul. Mr. Pintu Bhunia sincerely acknowledges the financial support received from UGC, Govt. of India in the form of Senior Research Fellowship under the mentorship of Prof Kallol Paul.}

\renewcommand{\subjclassname}{\textup{2020} Mathematics Subject Classification}
	
	\maketitle

	\begin{abstract}
		The two well-known numerical radius inequalities for the tensor product $A \otimes B$  acting on $\mathbb{H} \otimes \mathbb{K}$, where $A$ and $B$ are bounded linear operators defined  on complex Hilbert spaces $\mathbb{H} $ and $ \mathbb{K},$ respectively are, $ \frac{1}{2} \|A\|\|B\| \leq w(A \otimes B) \leq \|A\|\|B\| $ and $ 
			w(A)w(B) \leq w(A \otimes B) \leq \min \{ w(A) \|B\|, w(B) \|A\| \}. $ In this article we develop new lower and upper bounds for the numerical radius $w(A \otimes B)$  of the tensor product $A \otimes B $  and study the equality conditions for those bounds.

	\end{abstract}

\section{\textbf{Introduction}}
\noindent

The numerical range of a bounded linear operator on a complex Hilbert space has been an active area of research over a long period of time due to its application in different areas of pure and applied sciences. Development of bounds for the numerical radius (an important numerical constant associated with the numerical range) has attracted many mathematicians \cite{BBPMIA,BhuLAA21,BPADM,BPRIM,BPBDM,KIT05,KIT03,YamStu}  in recent years. The same for the tensor product of two operators acting on Hilbert spaces has been done by few mathematicians \cite{FosJMAA,GauLAA, GauLAMA, GAU}. In this article we focus on the development of bounds of the numerical radius of tensor product of two operators defined on  complex Hilbert spaces. Before proceeding further we introduce the notation and terminologies to be used throughout the paper. 

\smallskip
Let $A$ be a bounded linear operator on a complex Hilbert space $\mathbb{H}$ with inner product $\langle \cdot,\cdot\rangle$ and the induced norm $\|\cdot\|.$  Let $ \mathbb{B}(\mathbb{H})$ denote the $C^*$-algebra of all bounded linear operators defined on $\mathbb{H}.$ Let $A^*$ denote the adjoint of $A$ and $|A|$ denote the positive operator $({A^*A})^{1/2}.$ The Cartesian  decomposition of $A$ is  $A=\Re(A)+{\rm i}\Im(A),$ where $\Re(A)=\frac{A+A^*}{2}$ and $\Im(A)=\frac{A-A^*}{2\rm i},$  are known as  real  and imaginary part of $A,$ respectively.  The numerical range or the field of values of $A$ is defined as the range of the mapping $ x \mapsto \langle Tx,x \rangle$  on the unit sphere of $\mathbb{H}$ and is denoted by $W(A).$  From the famous Toeplitz-Hausdorff theorem it follows that  the numerical range  is always a convex set.   The numerical radius and Crawford number of $A,$ denoted as $w(A)$ and $c(A),$ respectively,  are defined as 
$	w(A)= \sup\left\{|\lambda|~:~ \lambda \in W(A)\right\}$
and $c(A)= \inf\left\{|\lambda|~:~ \lambda \in W(A)\right\}.$ The operator norm of $A$ is defined as $\|A\|=\{ \|Ax\|~:~ x \in \mathbb{H},~ \|x\|=1\}.$ It is well known (see \cite{YamStu}) that $w(A)=\sup_{\theta \in \mathbb{R}}\|\Re(e^{i \theta}A)\|=\sup_{\theta \in \mathbb{R}}\|\Im(e^{i \theta}A)\|$. It is easy to check that the numerical radius $w(\,\cdot\,) $ is a norm on  $\mathbb{B}(\mathbb{H}),$ and is equivalent to the operator norm that satisfies the inequality 
$\frac{1}{2}\|A\| \leq w(A) \leq \|A\|.$
The first inequality becomes equality if $A^2=0$ and  the second inequality becomes equality if $A^*A=AA^*$. In recent times this inequality has been improved using different techniques. Interested readers can  see \cite{BBPMIA,BhuLAA21,BPADM,BPRIM,KIT05,KIT03} and the references therein. 

Next we focus our attention to the tensor product of two complex Hilbert spaces $\mathbb{H}$ and $\mathbb{K},$   which is defined as the completion of the inner product space consisting of all elements of the form $\sum_{i=1}^{n}x_i \otimes y_i$ for $ x_i \in \mathbb{H}$ and $ y_i \in \mathbb{K}$, for $n\geq 1$, under the  inner product $\langle x \otimes y, z \otimes w \rangle=\langle x , z\rangle \langle y , w\rangle.$  The tensor product of the spaces $\mathbb{H}$ and $\mathbb{K}$ is denoted by $\mathbb{H} \otimes \mathbb{K}.$ Here the expression $x \otimes y$ is defined algebraically so as to be bilinear in the two arguments $x$ and $y.$
  The tensor product of two operators $A$ on $\mathbb{H}$ and $B$ on $\mathbb{K},$ denoted by $A \otimes B, $ is defined as $(A \otimes B)(x \otimes y)=Ax \otimes By$ for $x \otimes y \in \mathbb{H} \otimes \mathbb{K}.$ For any $A \in \mathbb{B}(\mathbb{H})$ and $B \in \mathbb{B}(\mathbb{K}),$ $A \otimes B \in \mathbb{B}(\mathbb{H} \otimes \mathbb{K})$ as it satisfies $\|A \otimes B\|=\|A\|\|B\|$.
 Therefore, any $A \otimes B \in \mathbb{B}(\mathbb{H} \otimes \mathbb{K})$ satisfies the following inequality
     \begin{eqnarray}\label{eqn1}
     	\frac{1}{2} \|A\|\|B\| \leq w(A \otimes B) \leq \|A\|\|B\|.
     \end{eqnarray}
Observe that the constants $\frac12$ and $1$ are best possible.  $\frac{1}{2} \|A\|\|B\| = w(A \otimes B)$ if $A^2=0$, $B$ is normal and  $ w(A \otimes B) = \|A\|\|B\|$ if $A$, $B$ are normal. We also note the following well-known inequality
 \begin{eqnarray}\label{eqn10}
	w(A)w(B) \leq w(A \otimes B) \leq \min \{ w(A) \|B\|, w(B) \|A\| \}.
\end{eqnarray}  
First inequality follows from the fact that $ w(A \otimes B) \geq |\langle(A \otimes B)x\otimes y, x\otimes y  \rangle|= |\langle Ax,x\rangle| \, | \langle By,y\rangle|$ with $\| x\|=\|y\|=1$. Following \cite[Th. 3.4]{double}, the second inequality follows from the fact that $A \otimes B= (A \otimes I_{\mathbb{K}})(I_{\mathbb{H}} \otimes B)$ ($I_{\mathbb{H}}$ and $I_{\mathbb{K}}$ are the identity operators on $\mathbb{H}$ and $\mathbb{K}$, respectively) and the two operators $A \otimes I_{\mathbb{K}}$ and $I_{\mathbb{H}} \otimes B$  double commute, that is, $A \otimes I_{\mathbb{K}}$ commutes with both $I_{\mathbb{H}} \otimes B$ and its adjoint $I_{\mathbb{H}} \otimes B^*$.
The authors in \cite{GAU} studied various equality conditions of the above inequalities in \eqref{eqn10}. We see that the first inequality in \eqref{eqn1} is not comparable, in general, with the first inequality in \eqref{eqn10}.  
In this paper,  we obtain many improvements of the first inequality in \eqref{eqn1}. We also give a complete characterization for the equality of $ w(A \otimes B)=\frac{1}{2} \|A\|\|B\|$. Other equality conditions are also studied. Further, we obtain various  refinements of the second inequality in \eqref{eqn1}.
	
\section{\textbf{Main results}}

We begin this section by  noting that for the tensor product of two operators $A \otimes B ,$ the numerical radius of $A \otimes B$ satisfies the following inequality
\begin{eqnarray}\label{eqna}
    w(A)w(B) \leq w(A \otimes B) \leq 2w(A)w(B).
\end{eqnarray}
Observe that the scalars $1$ and $2$ are the best possible constants. If $A$ or $B$ is normal then $w(A)w(B) = w(A \otimes B)$ and if  $A^2=B^2=0$ then $w(A \otimes B)=2w(A)w(B).$ 
We first obtain an upper bound for the numerical radius of $A \otimes B$ which improves the second inequality in \eqref{eqna}. For this purpose,  we define the numerical radius distance $d(A)$ of the operator $A$ from the scalar operators,
 which is defined as $d(A)=\inf \{ w(A-\lambda I) : {\lambda \in \mathbb{C}} \}.$ Now, we are in a position to prove the following improvement.

\begin{theorem}\label{theo6}
	Let $A \otimes B \in \mathbb{B}(\mathbb{H} \otimes \mathbb{K}).$ Then 
	\begin{eqnarray*}
		w(A \otimes B) \leq \min\left\{w(A)\Big(w(B)+d(B)\Big), w(B)\Big(w(A)+d(A)\Big)\right\} \leq 2w(A)w(B).
	\end{eqnarray*}
\end{theorem}

\begin{proof}
	It follows from the compactness argument that 
	there exists $\lambda_0 \in \mathbb{C}$ such that $w(B-\lambda_0 I)=d(B).$ If $\lambda_0=0,$ then $w(A \otimes B) \leq w(A)\left(w(B)+d(B)\right) = 2w(A)w(B).$ Now, let $\lambda_0 \neq 0$, consider  $\mu=\frac{\overline{\lambda_0}}{|\lambda_0|}$. Then we get
	\begin{eqnarray}
	w(A \otimes B)  &=& w(A \otimes (\mu B))\nonumber\\
	&=& w\left(A \otimes \Re(\mu B) +iA \otimes \Im(\mu B)\right)\nonumber\\
	&\leq& w\left(A \otimes \Re(\mu B)\right)+w\left(A \otimes \Im(\mu B)\right)\nonumber\\
	&=& w(A)\left(\|\Re(\mu B)\|+\|\Im(\mu B)\|\right) \,\, \Big(\text{$\Re(\mu B)$, $\Im(\mu B)$ are selfadjoint}\Big)\nonumber\\
	&=& w(A)\left(\|\Re(\mu B)\|+\|\Im(\mu (B-\lambda_0 I))\|\right)\nonumber\\
	&\leq& w(A)\left(w(\mu B)+w(\mu (B-\lambda_0 I))\right)\nonumber\\
	&=& w(A)\left(w(B)+w(B-\lambda_0 I)\right).\nonumber
\end{eqnarray}
Hence
	\begin{eqnarray}\label{theo6eqn1}
		w(A \otimes B) & \leq & w(A)\left(w(B)+d(B)\right).
	\end{eqnarray}
Similarly, we can prove that 
	\begin{eqnarray}\label{theo6eqn2}
	w(A \otimes B) & \leq&  w(B)\left(w(A)+d(A)\right).
	\end{eqnarray}
	Now, combining (\ref{theo6eqn1}) and (\ref{theo6eqn2}) we get the desired first inequality. The second inequality follows from $d(A)\leq w(A)$ and $d(B)\leq w(B).$
\end{proof}

\begin{remark}
	
For any operator $A\in \mathbb{B}(\mathbb{H})$, clearly $d(A)=w(A)$ if and only if $A$ is numerical radius orthogonal to $I$ in the sense of Birkhoff-James.  Therefore, it follows from Theorem \ref{theo6} that if $w(A \otimes B) = 2w(A)w(B)$ then both $A$ and $B$ both are  numerical radius orthogonal to $I$ in the sense of Birkhoff-James.  For more on  numerical radius orthogonality we refer to \cite{MPSMon}.	

\end{remark}

%Here $I$ denotes the identity operator on $\mathbb{H}.$

To prove next result we need the following sequence of lemmas.

\begin{lemma}$($\cite{KIT1988}$)$.\label{lemma1}
	Let $A\in \mathbb{B}(\mathbb{H})$ be positive, and let $x\in \mathbb{H}$ with $\|x\|=1.$ Then 
	\begin{eqnarray*}
		\langle Ax,x\rangle^r\leq \langle A^rx,x\rangle
	\end{eqnarray*}
	for all $r \geq 1.$ 
\end{lemma}

\begin{lemma}$($\cite{HalBook82}$)$.\label{lemma2}
	Let $A,B\in \mathbb{B}(\mathbb{H})$, and  let $x, y\in \mathbb{H}$. Then 
	\[|\langle Ax,x\rangle|\leq \langle |A|x,x\rangle^{1/2}\langle |A^*|x,x\rangle^{1/2}.\] 
\end{lemma}

\begin{lemma}$($\cite{KITJFA1997}$)$.\label{lemma3}
	Let $A,B\in \mathbb{B}(\mathbb{H})$ be positive. Then 
	\[\|A+B\|\leq \max\{\|A\|,\|B\|\}+\left\|A^{1/2}B^{1/2}\right\|.\]	 
\end{lemma}

Now, we obtain the following upper bounds for $w( A \otimes B)$.

\begin{theorem}\label{theo1a}
	Let $A \otimes B \in \mathbb{B}(\mathbb{H} \otimes \mathbb{K}).$ Then 
	\begin{eqnarray*}
		w^2(A \otimes B)&& \leq \frac{1}{4} \left \| \, |A|^2 \otimes |B|^2+|A^*|^2 \otimes |B^*|^2\, \right\|+\frac{1}{2}\left \|\, \Re(|A||A^*| \otimes |B||B^*|)\, \right \|\\
		&& \leq \frac{1}{4}\left(\|A\|\|B\|+\|A^2\|^{1/2}\|B^2\|^{1/2}\right)^2\\
		&& \leq \|A\|^2\|B\|^2.
	\end{eqnarray*}
\end{theorem}

\begin{proof}
		Let $f \in \mathbb{H} \otimes \mathbb{K}$ with $\|f\|=1.$ Now, using Lemma \ref{lemma2} we get
		\begin{eqnarray*}
			\left|\left\langle (A \otimes B)f,f \right\rangle \right|
			&& \leq \left\langle |A \otimes B|f,f \right\rangle^{1/2} \left\langle |A^* \otimes B^*| f,f \right\rangle^{1/2}\\
			&& \leq \frac{1}{2}\left\langle \left(|A \otimes B|+|A^* \otimes B^*|\right)f,f \right\rangle  
		\end{eqnarray*}
	Therefore, we have
	\begin{eqnarray*}
		\left|\left\langle (A \otimes B)f,f \right\rangle \right|^2
		&& \leq \frac{1}{4}\left\langle \left(|A \otimes B|+|A^* \otimes B^*|\right) f,f \right\rangle^{2}\\
		&& \leq \frac{1}{4}\left\langle \left(|A \otimes B|+|A^* \otimes B^*|\right)^2 f,f \right\rangle
		\,\,\Big(\mbox{by Lemma \ref{lemma1}}\Big)\\
		&& = \frac{1}{4}\left\langle \left(|A|^2 \otimes |B|^2+|A^*|^2 \otimes |B^*|^2\right)f,f \right\rangle\\
		&&  \,\,\,\,\,\, +\frac{1}{2}\left\langle \Re(|A||A^*| \otimes |B||B^*|)f,f \right\rangle\\
		&& \leq \frac{1}{4}\left\| \, |A|^2 \otimes |B|^2+|A^*|^2 \otimes |B^*|^2\|+\frac{1}{2}\|\Re(|A||A^*| \otimes |B||B^*|) \, \right \|.
	\end{eqnarray*}
   Taking supremum over all  $f \in \mathbb{H} \otimes \mathbb{K}$ with $\|f\|=1$ we get the first inequality.
   Now, clearly $\| \, |A||A^*| \otimes |B||B^*| \, \|=\|A^2\|\|B^2\|$ and using Lemma \ref{lemma3} we have
   	\begin{eqnarray}\label{theo1aeqn1}
   	\frac{1}{4}\left \| \, |A|^2 \otimes |B|^2+|A^*|^2 \otimes |B^*|^2\, \right \| && \leq \frac{1}{4}\left(\| \, |A|^2 \otimes |B|^2\|+\||A||A^*| \otimes |B||B^*| \, \|\right)\nonumber\\
   	&&= \frac{1}{4}\left(\|A\|^2\|B\|^2+\|A^2\|\|B^2\|\right).
   	\end{eqnarray}
   	Also,
   	\begin{eqnarray}\label{theo1aeqn2}
   	  \frac{1}{2}\|\Re(|A||A^*| \otimes |B||B^*|)\| \leq \frac{1}{2}\|A^2\|\|B^2\| \leq \frac{1}{2}\|A\|\|A^2\|^{1/2}\|B\|\|B^2\|^{1/2}.
   	\end{eqnarray}
   	Combining (\ref{theo1aeqn1}) and (\ref{theo1aeqn2}) we get the second inequality. The third inequality follows trivially.
\end{proof}

From Theorem \ref{theo1a} and the inequality \eqref{eqn1}  it follows that  $w(A \otimes B)=\frac12 \|A\| \|B\|$ if $A^2=B^2=0.$
Next bound reads as follows.

\begin{theorem}\label{theo1b}
	Let $A \otimes B \in \mathbb{B}(\mathbb{H} \otimes \mathbb{K}).$ Then 
	\begin{eqnarray*}
		w^2(A \otimes B)
		\leq \frac{1}{2}\|A^*A \otimes B^*B+AA^* \otimes BB^*\| \leq \|A\|^2\|B\|^2.
	\end{eqnarray*}
\end{theorem}

\begin{proof}
	Let $f \in \mathbb{H} \otimes \mathbb{K}$ with $\|f\|=1.$ Then by the Cartesian decomposition of $A\otimes B,$ we get
	\begin{eqnarray*}
		\left|\left\langle (A \otimes B) f,f \right\rangle \right|^2
		&&=\left \langle \Re(A \otimes B) f,f\right\rangle^2+\left\langle \Im(A \otimes B)f,f\right\rangle^2\\
		&& \leq \left\| \Re(A \otimes B)f\right\|^2+\left\| \Im(A \otimes B)f\right\|^2\\
		&&=\left\langle \Re^2(A \otimes B)f,f\right\rangle+\left\langle \Im^2(A \otimes B)f,f\right\rangle\\
		&&=\left\langle \left(\Re^2(A \otimes B)+\Im^2(A \otimes B)\right)f,f\right\rangle\\
		&& \leq \|\Re^2(A \otimes B)+\Im^2(A \otimes B)\|\\
		&& =\frac{\|A^*A \otimes B^*B+AA^* \otimes BB^*\|}{2}.
	\end{eqnarray*}
	Therefore, taking supremum over all  $f \in \mathbb{H} \otimes \mathbb{K}$ with $\|f\|=1$ we get the first inequality. The second inequality follows from 
	the triangle inequality of the operator norm $\|\cdot \|$. 
\end{proof}

	Next, we obtain a lower bound for $	w(A \otimes B).$

	\begin{theorem}\label{theo1}
		Let $A \otimes B \in \mathbb{B}(\mathbb{H} \otimes \mathbb{K}).$ Then 
		\begin{eqnarray*}
			w(A \otimes B) \geq \frac{1}{2}\|A\|\|B\|+ \frac{1}{4}\left|\|A \otimes B+A^* \otimes B^*\|-\|A \otimes B-A^* \otimes B^*\|\right|. 
		\end{eqnarray*}
	\end{theorem}
	
	\begin{proof}
		Let $f \in \mathbb{H} \otimes \mathbb{K}$ with $\|f\|=1.$ Then by the Cartesian decomposition of $A\otimes B,$ we get
		\begin{eqnarray*}
			\left|\left\langle (A \otimes B) f,f\right \rangle \right|^2
			=\left\langle \Re(A \otimes B)f,f\right\rangle^2+\left\langle \Im(A \otimes B)f,f\right\rangle^2.
		\end{eqnarray*}
	    Therefore, we have
	    \begin{eqnarray*}
	    	\left|\left\langle (A \otimes B)f,f \right\rangle\right| \geq \left|\left\langle \Re(A \otimes B)f,f\right\rangle\right|.
	    \end{eqnarray*}
    Taking supermum over all  $f \in \mathbb{H} \otimes \mathbb{K}$ with $\|f\|=1$ we get, 
    \begin{eqnarray}\label{theo1 eqn1}
    	w(A \otimes B) \geq \|\Re(A \otimes B)\|.
    \end{eqnarray}
    Also, 
     \begin{eqnarray*}
    	\left|\left\langle (A \otimes B)f,f \right\rangle\right| \geq \left|\left\langle \Im(A \otimes B)f,f\right\rangle\right|.
    \end{eqnarray*}
    Taking supermum over all  $f \in \mathbb{H} \otimes \mathbb{K}$ with $\|f\|=1$ we get, 
    \begin{eqnarray}\label{theo1 eqn2}
      w(A \otimes B) \geq \|\Im(A \otimes B)\|.
    \end{eqnarray}
    Combining (\ref{theo1 eqn1}) and (\ref{theo1 eqn2}) we get, 
    \begin{eqnarray*}
      w(A \otimes B)& \geq & \max\left\{\|\Re(A \otimes B)\|,\|\Im(A \otimes B)\|\right\}\\
      &=& \frac{\|\Re(A \otimes B)\|+\|\Im(A \otimes B)\|}{2}+\frac{\left|\|\Re(A \otimes B)\|-\|\Im(A \otimes B)\|\right|}{2}\\
      & \geq & \frac{\|\Re(A \otimes B)+i\Im(A \otimes B)\|}{2}+\frac{\left|\|\Re(A \otimes B)\|-\|\Im(A \otimes B)\|\right|}{2}\\
      &=& \frac{\|A\|\|B\|}{2}+ \frac{\left|\|A \otimes B+A^* \otimes B^*\|-\|A \otimes B-A^* \otimes B^*\|\right|}{4}.
    \end{eqnarray*}
    Therefore, we get the desired inequality.
	\end{proof}
	
	As a consequence of the above theorem we have  the following corollary.
	
\begin{cor}\label{theo1cor1}
	Let $A \otimes B \in \mathbb{B}(\mathbb{H} \otimes \mathbb{K}).$ If the equality $w(A \otimes B)=\frac{\|A\|\|B\|}{2}$ holds then 
	$\|A \otimes B+A^* \otimes B^*\|=\|A \otimes B-A^* \otimes B^*\|=\|A\|\|B\|.$
\end{cor}
	
\begin{proof}
	If  $w(A \otimes B)=\frac{\|A\|\|B\|}{2},$ then from Theorem \ref{theo1} we get
	$\|A \otimes B+A^* \otimes B^*\|=\|A \otimes B-A^* \otimes B^*\|.$ Now, 
	\begin{eqnarray*}
		\|A \otimes B+A^* \otimes B^*\| & \leq & 2w(A \otimes B) \\
		&=& \|A\|\|B\|\\
		& \leq & \frac{\|A \otimes B+A^* \otimes B^*\|+\|A \otimes B-A^* \otimes B^*\|}{2}\\
		&=& \|A \otimes B+A^* \otimes B^*\|.
	\end{eqnarray*}
    So, we get the desired equalities
    $\|A \otimes B+A^* \otimes B^*\|=\|A \otimes B-A^* \otimes B^*\|=\|A\|\|B\|.$
\end{proof}	
	
Next, we obtain a complete characterization for the equality of $w(A \otimes B)=\frac{\|A\|\|B\|}{2}$. 	

\begin{proposition}\label{theo1prop1}
	Let $A \otimes B \in \mathbb{B}(\mathbb{H} \otimes \mathbb{K}).$ Then the equality  $w(A \otimes B)=\frac{\|A\|\|B\|}{2}$ holds if and only if 	$\| e^{i\theta} A \otimes B+e^{-i\theta}A^* \otimes B^*\|=\|e^{i\theta}A \otimes B-e^{-i\theta}A^* \otimes B^*\|=\|A\|\|B\|,$ for all $\theta \in \mathbb{R}.$
\end{proposition}

\begin{proof}
	The sufficient part of the proposition follows easily, we only prove the necessary part.
	 If the equality $w(A \otimes B)=\frac{\|A\|\|B\|}{2}$ holds then form Corollary \ref{theo1cor1} we get $\|A \otimes B+A^* \otimes B^*\|=\|A \otimes B-A^* \otimes B^*\|=\|A\|\|B\|.$ As for all  $\theta \in \mathbb{R},$ $e^{i\theta}A \otimes B \in \mathbb{B}(\mathbb{H} \otimes \mathbb{K})$ and $w(A \otimes B)=w(e^{i\theta}A \otimes B).$ Therefore we have  for all $\theta \in \mathbb{R},$ $\| e^{i\theta} A \otimes B+e^{-i\theta}A^* \otimes B^*\|=\|e^{i\theta}A \otimes B-e^{-i\theta}A^* \otimes B^*\|=\|A\|\|B\|,$ as desired.
	 
\end{proof}	

In the following theorem, we obtain another lower bound for $w(A \otimes B)$ which is incomparable with the bound  obtained in Theorem \ref{theo1}.

 \begin{theorem}\label{theo3}
	Let $A \otimes B \in \mathbb{B}(\mathbb{H} \otimes \mathbb{K}).$ Then 
	\begin{eqnarray*}
		w(A \otimes B) \geq \frac{\|A\|\|B\|}{2}+ \frac{\left | \, \|A \otimes B+iA^* \otimes B^*\|-\|A \otimes B-iA^* \otimes B^*\| \, \right | }{4}.
	\end{eqnarray*}
\end{theorem}

\begin{proof}
	Let $f \in \mathbb{H} \otimes \mathbb{K}$ with $\|f\|=1.$ Then by the Cartesian decomposition of $A\otimes B,$ we get
	\begin{eqnarray*}
		\left|\left\langle (A \otimes B)f,f \right\rangle\right|^2
		&&=\left\langle \Re(A \otimes B)f,f\right\rangle^2+\left\langle \Im(A \otimes B)f,f\right\rangle^2\\
		&& \geq \frac{1}{2}\left(\left|\left\langle \Re(A \otimes B)f,f\right\rangle\right|
		+\left|\left\langle \Im(A \otimes B)f,f\right\rangle\right|\right)^2\\
		&& \geq \frac{1}{2}\left|\left\langle (\Re(A \otimes B) \pm \Im(A \otimes B))f,f\right\rangle\right|^2.
	\end{eqnarray*}
	Therefore, 
	\begin{eqnarray*}
		\left|\left\langle (A \otimes B)f,f \right\rangle\right|
		\geq \frac{1}{\sqrt{2}}\left|\left\langle (\Re(A \otimes B) \pm \Im(A \otimes B))f,f\right\rangle\right|.
	\end{eqnarray*}
	Taking supermum over all  $f \in \mathbb{H} \otimes \mathbb{K}$ with $\|f\|=1$ we get,
	\begin{eqnarray}\label{theo3 eqn1}
		w(A \otimes B) \geq \frac{\|\Re(A \otimes B) \pm \Im(A \otimes B)\|}{\sqrt{2}}.
	\end{eqnarray}
Now, it follows from the inequalities in \eqref{theo3 eqn1} that
	\begin{eqnarray*}
		w(A \otimes B)
		&& \geq \frac{1}{\sqrt{2}}\max\left\{\|\Re(A \otimes B) + \Im(A \otimes B)\|,\|\Re(A \otimes B) - \Im(A \otimes B)\|\right\}\\
		&&= \frac{\|\Re(A \otimes B) + \Im(A \otimes B)\|+\|\Re(A \otimes B)- \Im(A \otimes B)\|}{2\sqrt{2}}\\
		&&\,\,\,\,\,+\frac{\left|\|\Re(A \otimes B) + \Im(A \otimes B)\|-\|\Re(A \otimes B) - \Im(A \otimes B)\|\right|}{2\sqrt{2}}\\ 
		&& \geq  \frac{\|(\Re(A \otimes B) + \Im(A \otimes B))+i(\Re(A \otimes B)- \Im(A \otimes B))\|}{2\sqrt{2}}\\
		&&\,\,\,\,\,+\frac{\left|\|\Re(A \otimes B) + \Im(A \otimes B)\|-\|\Re(A \otimes B) - \Im(A \otimes B)\|\right|}{2\sqrt{2}}\\ 
		&&= \frac{\|(1+i)A^* \otimes B^*\|}{2\sqrt{2}}\\
		&& \,\,\,\,\,+\frac{\left|\|\Re(A \otimes B) + \Im(A \otimes B)\|-\|\Re(A \otimes B) - \Im(A \otimes B)\|\right|}{2\sqrt{2}}\\
		&&= \frac{\|A\|\|B\|}{2}+ \frac{\left| \, \|A \otimes B+iA^* \otimes B^*\|-\|A \otimes B-iA^* \otimes B^*\| \, \right|}{4}.
	\end{eqnarray*}
	Therefore, we get the desired inequality.
\end{proof}

\begin{remark}
	If the equality $w(A \otimes B)=\frac{\|A\|\|B\|}{2}$ holds then $\|A \otimes B+iA^* \otimes B^*\|=\|A \otimes B-iA^* \otimes B^*\|.$ It should be mentioned here that the converse is not true.
\end{remark}

Next lower bound reads as follows.	
	
\begin{theorem}\label{theo2}
	Let $A \otimes B \in \mathbb{B}(\mathbb{H} \otimes \mathbb{K}).$ Then 
	\begin{eqnarray*}
		 w^2(A \otimes B)
		& \geq & \frac{\|A^*A \otimes B^*B+AA^* \otimes BB^*\|}{4}\\
		&& + \frac{\left| \, \|A \otimes B+A^* \otimes B^*\|^2-\|A \otimes B-A^* \otimes B^*\|^2 \, \right|}{8}.
	\end{eqnarray*}
\end{theorem}	
	
	\begin{proof}
	 From the inequalities in (\ref{theo1 eqn1}) and (\ref{theo1 eqn2}) we obtain that
	\begin{eqnarray*}
		w^2(A \otimes B)
		&& \geq  \max\left\{\|\Re(A \otimes B)\|^2,\|\Im(A \otimes B)\|^2\right\}\\
		&&= \frac{\|\Re(A \otimes B)\|^2+\|\Im(A \otimes B)\|^2}{2}+\frac{\left|\|\Re(A \otimes B)\|^2-\|\Im(A \otimes B)\|^2\right|}{2}\\
		&& \geq \frac{\|\Re^2(A \otimes B)+\Im^2(A \otimes B)\|}{2}+\frac{\left|\|\Re(A \otimes B)\|^2-\|\Im(A \otimes B)\|^2\right|}{2}\\
		&& = \frac{\|A^*A \otimes B^*B+AA^* \otimes BB^*\|}{4}\\
		&& \,\,\,\,\,+ \frac{\left| \, \|A \otimes B+A^* \otimes B^*\|^2-\|A \otimes B-A^* \otimes B^*\|^2 \, \right|}{8}.
	\end{eqnarray*}
	Therefore, we get the desired inequality.
\end{proof}

Now, we give a complete characterization for the equality of $w^2(A \otimes B)=\frac{1}{4}\|A^*A \otimes B^*B+AA^* \otimes BB^*\|$.

\begin{proposition}\label{theo2prop1}
	Let $A \otimes B \in \mathbb{B}(\mathbb{H} \otimes \mathbb{K}).$ Then the equality  $$w^2(A \otimes B)=\frac{1}{4}\|A^*A \otimes B^*B+AA^* \otimes BB^*\|$$ holds if and only if 	$$\| e^{i\theta} A \otimes B+e^{-i\theta}A^* \otimes B^*\|^2=\|e^{i\theta}A \otimes B-e^{-i\theta}A^* \otimes B^*\|^2=\|A^*A \otimes B^*B+AA^* \otimes BB^*\|,$$ for all $\theta \in \mathbb{R}.$
\end{proposition}

\begin{proof}
	The sufficient part follows easily, we only prove the necessary part.
	Let $w^2(A \otimes B)=\frac{1}{4}\|A^*A \otimes B^*B+AA^* \otimes BB^*\|$. Then for any real number  $\theta$, we have
	\begin{eqnarray*}
		&& \|A^*A \otimes B^*B+AA^* \otimes BB^*\| \\
		&& = \frac{1}{2} \| (e^{i\theta} A \otimes B+e^{-i\theta}A^* \otimes B^*)^2+(e^{i\theta}A \otimes B-e^{-i\theta}A^* \otimes B^*)^2\|\\
		&& \leq \frac{1}{2} \left( \| e^{i\theta} A \otimes B+e^{-i\theta}A^* \otimes B^*\|^2+\|e^{i\theta}A \otimes B-e^{-i\theta}A^* \otimes B^*\|^2\right)\\
		&& \leq 4w^2(A \otimes B)\\
		&& =\|A^*A \otimes B^*B+AA^* \otimes BB^*\|.
	\end{eqnarray*}
    Thus, we conclude that
    $$\| e^{i\theta} A \otimes B+e^{-i\theta}A^* \otimes B^*\|^2=\|e^{i\theta}A \otimes B-e^{-i\theta}A^* \otimes B^*\|^2=\|A^*A \otimes B^*B+AA^* \otimes BB^*\|,$$ for all $\theta \in \mathbb{R}.$
    
\end{proof}

%...........................

Now, in the following theorem we obtain a lower bound for $w(A \otimes B)$ which is incomparable with the bound obtained in Theorem \ref{theo2}.

\begin{theorem}\label{theo4}
	Let $A \otimes B \in \mathbb{B}(\mathbb{H} \otimes \mathbb{K}).$ Then 
	\begin{eqnarray*}
		 w^2(A \otimes B)
		& \geq & \frac{\|A^*A \otimes B^*B+AA^* \otimes BB^*\|}{4}\\
		 && + \frac{\left| \, \|A \otimes B+iA^* \otimes B^*\|^2-\|A \otimes B-iA^* \otimes B^*\|^2 \, \right|}{8}.
	\end{eqnarray*}
\end{theorem}

\begin{proof}
	It follows from the inequalities in \eqref{theo3 eqn1} that 
  \begin{eqnarray*}
      w^2(A \otimes B)
      && \geq \frac{1}{2}\max\left\{\|\Re(A \otimes B) + \Im(A \otimes B)\|^2,\|\Re(A \otimes B) - \Im(A \otimes B)\|^2\right\}\\
      &&= \frac{\|\Re(A \otimes B) + \Im(A \otimes B)\|^2+\|\Re(A \otimes B)- \Im(A \otimes B)\|^2}{4}\\
      && \,\,\,\,\,+\frac{\left|\|\Re(A \otimes B) + \Im(A \otimes B)\|^2-\|\Re(A \otimes B) - \Im(A \otimes B)\|^2\right|}{4}\\ 
      && \geq  \frac{\|(\Re(A \otimes B) + \Im(A \otimes B))^2+(\Re(A \otimes B)- \Im(A \otimes B))^2\|}{4}\\
      &&\,\,\,\,\,+\frac{\left|\|\Re(A \otimes B) + \Im(A \otimes B)\|^2-\|\Re(A \otimes B) - \Im(A \otimes B)\|^2\right|}{4}\\ 
      &&= \frac{\|\Re^2(A \otimes B) + \Im^2(A \otimes B)\|}{2}\\
      &&\,\,\,\,\,+\frac{\left|\|\Re(A \otimes B) + \Im(A \otimes B)\|^2-\|\Re(A \otimes B) - \Im(A \otimes B)\|^2\right|}{4}\\
      &&= \frac{\|A^*A \otimes B^*B+AA^* \otimes BB^*\|}{4}\\
      &&\,\,\,\,\,+ \frac{\left| \, \|A \otimes B+iA^* \otimes B^*\|^2-\|A \otimes B-iA^* \otimes B^*\|^2 \, \right|}{8}.
  \end{eqnarray*}
Thus, we get the desired inequality.
\end{proof}

\begin{remark}
	If the equality $w^2(A \otimes B)=\frac{\|A^*A \otimes B^*B+AA^* \otimes BB^*\|}{4}$ holds then $\|A \otimes B+iA^* \otimes B^*\|=\|A \otimes B-iA^* \otimes B^*\|.$ However, the converse is not necessarily true.
\end{remark}

Finally, we obtain the following inequality.	
	
\begin{theorem}\label{theo5}
	Let $A \otimes B \in \mathbb{B}(\mathbb{H} \otimes \mathbb{K}).$ Then 
	\begin{eqnarray*}
		\|A \otimes B\|^2-w^2(A \otimes B)
		\leq \inf_{\lambda \in \mathbb{C}}\left\{\|A \otimes B-\lambda I \otimes I\|^2-c^2(A \otimes B-\lambda I \otimes I)\right\}.
	\end{eqnarray*}
\end{theorem}

\begin{proof}
		Let $f \in \mathbb{H} \otimes \mathbb{K}$ with $\|f\|=1.$ Then for any $\lambda \in \mathbb{C},$ we have
		\begin{eqnarray*}
			 \left\|(A \otimes B)f\right\|^2-\left|\left\langle (A \otimes B) f,f\right \rangle \right|^2
			&&= \left\|(A \otimes B-\lambda I \otimes I)f\right\|^2\\
			&& \,\,\,\,\,\,\,-\left|\left\langle (A \otimes B-\lambda I \otimes I)f,f\right\rangle\right|^2\\
			&& \leq \|A \otimes B-\lambda I \otimes I\|^2-c^2(A \otimes B-\lambda I \otimes I).
		\end{eqnarray*}
	 Therefore, taking supremum over all  $f \in \mathbb{H} \otimes \mathbb{K}$ with $\|f\|=1$, we get
	 \begin{eqnarray}\label{theo5 eqn1}
	 	\|A \otimes B\|^2-w^2(A \otimes B)
	 	\leq \|A \otimes B-\lambda I \otimes I\|^2-c^2(A \otimes B-\lambda I \otimes I).
	 \end{eqnarray}
	 As the inequality (\ref{theo5 eqn1}) holds for all $\lambda \in \mathbb{C},$ so taking infimum  over all  $\lambda \in \mathbb{C}$
	 we get the required inequality.
\end{proof}

\bibliographystyle{amsplain}

\begin{thebibliography}{99}
	
	
	\bibitem{BBPMIA} S. Bag, P. Bhunia and K. Paul, Bounds of numerical radius of bounded linear operators using $t$-Aluthge transform, Math. Inequal. Appl., 23 (2020), no. 3, 991--1004.
	
	\bibitem{BhuLAA21} P. Bhunia and K. Paul, Development of inequalities and characterization of equality conditions for the numerical radius, Linear Algebra Appl., 630 (2021), 306--315.
	
	
	\bibitem{BPADM}  P. Bhunia and K. Paul, Furtherance of numerical radius inequalities of Hilbert space operators, Arch. Math. (Basel), 117 (2021), no. 5, 537--546.
	
	\bibitem{BPRIM} P. Bhunia and K. Paul, Proper improvement of well-known numerical radius inequalities and their applications. Results Math. 76 (2021), no. 4, Paper No. 177, 12 pp.
	
	\bibitem{BPBDM} P. Bhunia and K. Paul, New upper bounds for the numerical radius of Hilbert space operators, Bull. Sci. Math., 167 (2021), Paper No. 102959, 11 pp. 
	
	\bibitem{FosJMAA} A. Fošner, Z.  Huang, C.-K. Li, N.-S. Sze, Linear maps preserving numerical radius of tensor products of matrices, J. Math. Anal. Appl., 407 (2013), no. 2, 183--189.
	
	\bibitem{GauLAA} H.-L. Gau,  Y.-H. Lu, Extremality of numerical radii of tensor products of matrices,  Linear Algebra Appl.,  565 (2019), 82--98.
	
	\bibitem{GauLAMA} H.-L. Gau, K.-Z. Wang, P. Y. Wu, Numerical radii for tensor products of matrices,  Linear Multilinear Algebra,  63 (2015), no. 10, 1916--1936.
	
	\bibitem{GAU} H.-L. Gau, K.-Z. Wang and P.Y. Wu, Numerical radii for tensor products of operators, Integr. Equ. Oper. Theory, 78 (2014), 375--382.
	
	\bibitem{HalBook82} P.R. Halmos, A Hilbert Space Problems Book, Springer, New York (1982).
	
	\bibitem{double} J.A.R. Holbrook,  Multiplicative properties of the numerical radius in operator theory, J. Reine Angew. Math., 237 (1969), 166--174.
	
	\bibitem{KIT05}  F. Kittaneh, Numerical radius inequalities for Hilbert space operators, Studia Math., 168 (2005), no. 1, 73--80. 
	
	\bibitem{KIT03}  F. Kittaneh,  A numerical radius inequality and an estimate for the numerical radius of the Frobenius companion matrix, Studia Math., 158 (2003), no. 1, 11--17. 
	
	\bibitem{KITJFA1997} F. Kittaneh, Norm inequalities for certain operator sums. J. Funct. Anal. 143 (1997), 337--348.
	
	\bibitem{KIT1988} F. Kittaneh, Notes on some inequalities for Hilbert space operators, Publ. Res. Inst. Math. Sci. 24 (1988), 283--293.
	
	\bibitem{MPSMon} A. Mal, K. Paul and J. Sen, Birkhoff–James orthogonality and numerical radius inequalities of operator matrices, Monatsh. Math., (2022). https://doi.org/10.1007/s00605-021-01638-1 
	
	\bibitem{YamStu} T. Yamazaki, On upper and lower bounds for the numerical radius and an equality condition, Studia Math.,  178 (2007), no. 1, 83--89.
	
	
	
\end{thebibliography}

\end{document}